\documentclass{amsart}

\usepackage{amssymb,latexsym,amscd,amsthm,amssymb, amscd, amsfonts,enumerate,array,bbm,bm}

\newtheorem{theorem}{Theorem}
\newtheorem{proposition}[theorem]{Proposition}
\newtheorem{conjecture}[theorem]{Conjecture}

\newtheorem{lemma}[theorem]{Lemma}

\newtheorem{maintheorem}[theorem]{Main Theorem}

\theoremstyle{definition}

\newtheorem{example}[theorem]{Example}

\def\ZZ{\mathbb{Z}}
\def\QQ{\mathbb{Q}}

\begin{document}

\date{October 31, 2010}

\thanks{The authors were
supported in part
by the NSF grant DMS \#0800247}

\title[A short proof of Kontsevich cluster conjecture]
{A short proof of Kontsevich cluster conjecture}

\author{Arkady Berenstein}
\address{\noindent Department of Mathematics, University of Oregon,
Eugene, OR 97403, USA} \email{arkadiy@math.uoregon.edu}

\author{Vladimir Retakh}
\address{Department of Mathematics, Rutgers University}
\email{vretakh@math.rutgers.edu}

\maketitle

%\tableofcontents

%\section{introduction}

The aim of this note is to give an elementary proof of the following Kontsevich conjecture.

Recall that the {\it Kontsevich map} $K_r$, $r\in \ZZ_{>0}$ is the following (birational) automorphism of 
a noncommutative plane:
$$K_r:(x,y)\mapsto (xyx^{-1},(1+y^r)x^{-1}) \ ,$$
%where $H(y)=y^r+1+\sum_{1\le k\le r/2} c_k (y^k+y^{r-k})$ is a monic palindromic polynomial of degree $r$. 

%M. Kontsevich conjectured that 

\begin{conjecture} 
\label{conj:kontsevich} (M.~Kontsevich) For any $r_1,r_2\in \ZZ_{>0}$ all iterations $\underbrace{\cdots K_{r_1}K_{r_2}K_{r_1}}_k(x,y)$, $k\ge 1$ are given by noncommutative Laurent polynomials in $x$ and $y$. 

\end{conjecture}

%This conjecture was generalized in \cite{KeDif4} as follows.

%\begin{conjecture}
%\label{conj:kontsevich2}
%For any $r_1,r_2\in \ZZ_{>0}$ all iterations $\underbrace{\cdots K_{r_1}K_{r_2}K_{r_1}}_k(x,y)$, $k\ge 1$ are given by noncommutative Laurent polynomials in $x$ and $y$
%with non-negative integer coefficients.
%\end{conjecture} 

The Kontsevich conjecture was first proved for $r_1=r_2=2$ by A.~Usnich in \cite{Usnich} and was later settled by A.~Usnich in \cite{Usnich1} in greater generality when $r_1=r_2=r$  (with $1+y^r$ replaced by any monic palindromic polynomial $H(y)$) by means of derived 
categories.
Independently, Conjecture \ref{conj:kontsevich} was verified for $(r_1,r_2)\in \{(2,2),(4,1),(1,4)\}$ in \cite{KeDiF4} along with the positivity conjecture: 
for $(r_1,r_2)\in \{(2,2),(4,1),(1,4)\}$ all noncommutative Laurent polynomials in question have nonnegative integer coefficients. 

%It turns out that our noncommutative clusters on the cylinder  settle the {\it positive} Kontsevich conjecture due to the following result from our forthcoming paper \cite{BR2}.

Our goal is to give a short proof of Conjecture \ref{conj:kontsevich}.

\begin{theorem} 
\label{th:kontsevich} For any $r_1,r_2\in \ZZ_{>0}$ all iterations $\underbrace{\cdots K_{r_1}K_{r_2}K_{r_1}}_k(x,y)$, $k\ge 1$ are given by noncommutative Laurent polynomials in $x$ and $y$. 

\end{theorem}

%\section{Main result}

To present our proof of Theorem \ref{th:kontsevich}, we need some notation. Denote $$(x_k,y_k):=\underbrace{\cdots K_{r_1}K_{r_2}K_{r_1}}_k(x,y) $$
and denote $z:=[x,y]=xyx^{-1}y^{-1}$.
Then it is easy to see by induction that $[x_k,y_k]=[x,y]=z$ for all $k$. This taken together with the recursion $x_{k+1}=x_ky_kx_k^{-1}$ and $y_{k+1}=(1+y_k^{r_k})x_k^{-1}$, where 
\begin{equation}
\label{eq:r_k}r_k=
\begin{cases} 
r_1 &\text{if $k$ is odd} \\
r_2 &\text{if $k$ is even} \\
\end{cases}
\end{equation}
gives the following three recursions (they first appeared in \cite[Section 2.2]{KeDiF4})
$$x_{k+1}=zy_k,~y_{k+1}zy_{k-1}=1+y_k^{r_k}, ~y_{k+1}z y_k=y_ky_{k+1} \ .$$

Let ${\mathcal F}_2=\QQ\langle y_1^{\pm 1},y_2^{\pm 1}\rangle$ be the group algebra of the free group in $2$ generators. It was proved by A.I.~Malcev  (see e.g., \cite[Section 8.7]{Cohn}) that ${\mathcal F}_2$ is a 
{\it divisible algebra}, i.e., it embeds in a division ring (we denote the smallest one by $Frac({\mathcal F}_2)$). 

Define elements $y_k\in Frac({\mathcal F}_2)$, $k\in \ZZ\setminus \{1,2\}$
recursively by:  
\begin{equation}
\label{eq:cluster recursion} y_{k+1}zy_{k-1}=1+y_k^{r_k} \ ,
\end{equation} 
where $z:=[y_2^{-1},y_1]=y_2^{-1}y_1y_2y_1^{-1}$.  

Note that $y_0,y_3\in {\mathcal F}$ and let ${\mathcal A}={\mathcal A}(r_1,r_2)$ be the subalgebra of ${\mathcal F}$ generated by $y_0,y_1,y_2,y_3,z,z^{-1}$. 
We will  refer to ${\mathcal A}$ as a {\it (purely) noncommutative cluster algebra} of type $(r_1,r_2)$.

\begin{lemma} The elements $y_k\in Frac({\mathcal F}_2)$ satisfy for all $k\in \ZZ$:
\begin{equation}
\label{eq:z-commutation}
y_{k+1}z y_k=y_ky_{k+1}
\end{equation}
\end{lemma}
\begin{proof} Indeed, the \eqref{eq:z-commutation} is obvious for $k=1$. Let us prove it for $k\ge 1$ by induction. We will use the inductive hypothesis in the form 
$y_ky_{k-1}^{-1}z^{-1}=y_{k-1}^{-1}y_k$.
Indeed, since $y_{k+1}z=(1+y_k)^{r_k}y_{k-1}^{-1}$, we obtain
$$y_{k+1}z y_k-y_ky_{k+1}=(1+y_k)^{r_k}y_{k-1}^{-1}y_k-y_k(1+y_k)^{r_k}y_{k-1}^{-1}z^{-1}$$
$$=(1+y_k)^{r_k}y_{k-1}^{-1}y_k-(1+y_k)^{r_k}y_ky_{k-1}^{-1}z^{-1}=(1+y_k)^{r_k}y_{k-1}^{-1}y_k-(1+y_k)^{r_k}y_{k-1}^{-1}y_k=0$$
by the inductive hypothesis. The relation \eqref{eq:z-commutation} for $k\le 0$ also follows.
\end{proof}

Thus, based on the above discussion, Theorem \ref{th:kontsevich} directly follows from our main result.

\begin{maintheorem} 
\label{th:yk}
Each $y_k$ belongs to ${\mathcal A}$, e.g., $y_k$ is a noncommutative Laurent polynomial in $y_1,y_2$.

\end{maintheorem}

\begin{proof} 

Denote by ${\mathcal A}_k={\mathcal A}_k(r_1,r_2)$ the subalgebra of ${\mathcal F}_2$ generated by $y_k,y_{k+1},y_{k+2},y_{k+3},z^{\pm 1}$. It suffices to prove the following result (which is a noncommutative version of \cite[Formula (4.12)]{BFZ-cluster} and \cite[Lemma 5.8]{BZ10}).

\begin{theorem} 
\label{th:AkA}
${\mathcal A}_k={\mathcal A}$ for all $k\in \ZZ$.

\end{theorem}

\begin{proof} Since ${\mathcal A}={\mathcal A}_0$, it suffices to prove that ${\mathcal A}_k={\mathcal A}_{k+1}$
for $k\in \ZZ$, i.e., that for all $k\in \ZZ$ one has
%In fact, it will suffice to prove only one inclusion ${\mathcal A}_{k+1}\subseteq {\mathcal A}_k$, i.e., 
\begin{equation}
\label{eq:inclusionk} 
y_{k+4}\in {\mathcal A}_k,~y_k\in {\mathcal A}_{k+1}
\end{equation}
%for all $k\in \ZZ$ (the opposite inclusion $x_k\in {\mathcal A}_{k+1}$ can be proved by the identical argument). 

\begin{proposition} 
\label{pr:polynomial}
For each $n\in \ZZ$ one has: $y_{k+4}z=zy_k(y_{k+3}z)^{r_{k+1}}-\sum\limits_{j=0}^{r_{k+1}-1}(zy_{k+1})^jz(y_{k+2}z)^{r_k-1}(y_{k+3}z)^j$.
\end{proposition}

\begin{proof} For simplicity (and without loss of generality) we assume that $k=0$. We start with the following technical result.

\begin{lemma}
\label{le:reduction}
For each $m\ge 0$ we have: 
$
y_1^m(y_3z)^m=
1+\sum\limits_{k=0}^{m-1}y_1^k(y_2z)^{r_2}(y_3z)^k$.
\end{lemma}
\begin{proof}
We proceed by induction on $m$.
For $m=0$ the assertion is clear. Assume that $m>0$ and it holds for $m-1$. Let us prove it for $m$. 
Note that the  \eqref{eq:cluster recursion} and \eqref{eq:z-commutation} imply that 
\begin{equation}
\label{eq:cluster recursion2} 
y_{k-1}y_{k+1}z=1+(y_kz)^{r_k}  
\end{equation}
Indeed, using \eqref{eq:cluster recursion2}, we obtain 
$$y_1^m(y_3z)^m=y_1^{m-1}(y_1y_3z)(y_3z)^{m-1}=y_1^{m-1}(1+(y_2z)^{r_2})(y_3z)^{m-1}=y_1^{m-1}(y_2z)^{r_2}(y_3z)^{m-1}+y_1^{m-1}(y_3z)^{m-1}$$
$$=y_1^{m-1}(y_2z)^{r_2}(y_3z)^{m-1}+1+\sum_{k=0}^{m-2}y_1^k(y_2z)^{r_2}(y_3z)^k=1+\sum_{k=0}^{m-1}y_1^k(y_2z)^{r_2}(y_3z)^k \ .
$$
The lemma is proved.
\end{proof}

Furthermore, compute:
$$y_4z=y_2^{-1}((y_3z)^{r_1}+1)=y_2^{-1}(y_3z)^{r_1}+y_2^{-1}=(zy_0-y_2^{-1}(y_1)^{r_1})(y_3z)^{r_1}+y_2^{-1}$$
$$=zy_0(y_3z)^{r_1}-y_2^{-1}(y_1^{r_1-1}(y_1y_3z)(y_3z)^{r_1-1}-1)=zy_0(y_3z)^{r_1}-y_2^{-1}(y_1^{r_1-1}(1+(y_2z)^{r_2})(y_3z)^{r_1-1}-1).$$

We have: $$y_1^{r_1-1}(1+(y_2z)^{r_2})(y_3z)^{r_1-1}-1=y_1^{r_1-1}(y_2z)^{r_2}(y_3z)^{r_1-1}+y_1^{r_1-1}(y_3z)^{r_1-1}-1 \ .$$

Using Lemma \ref{le:reduction} and taking into account that $y_1^my_2=y_2(zy_1)^{m-1}$ for $m>0$,  we obtain:

$$y_1^{r_1-1}(1+(y_2z)^{r_2})(y_3z)^{r_1-1}-1= y_1^{r_1-1}(y_2z)^{r_2}(y_3z)^{r_1-1}+\sum_{k=0}^{r_1-2}y_1^k(y_2z)^{r_2}(y_3z)^k=\sum_{k=0}^{r_1-1}y_1^k(y_2z)^{r_2}(y_3z)^k$$
$$=y_2\sum_{k=0}^{r_1-1}(zy_1)^kz(y_2z)^{r_2-1}(y_3z)^k \ .$$

Therefore, $y_4z=zy_0(y_3z)^{r_1}-\sum\limits_{k=0}^{r_1-1}(zy_1)^kz(y_2z)^{r_2-1}(y_3z)^k$.
This proves Proposition \ref{pr:polynomial}.
\end{proof}

Proposition \ref{pr:polynomial} gives us the first inclusion \eqref{eq:inclusionk}. Prove second inclusion \eqref{eq:inclusionk} now.
We need the following obvious fact. Let $\sigma$ be the anti-automorphism of ${\mathcal F}_2$ given by: $\sigma(y_1)= y_2$, $\sigma(y_1)= y_2$ (so that $\sigma(z)= z$). 

\begin{lemma} \label{le:reflection}
 $\sigma(y_k)=y_{3-k}$  for $k\in \ZZ$, in particular, $\sigma({\mathcal A}_k(r_1,r_2))={\mathcal A}_{-k}(r_2,r_1)$ for $k\in \ZZ$.
\end{lemma}
This immediately implies the second inclusion \eqref{eq:inclusionk}: $y_{1-k}\in {\mathcal A}_{-k}$,  $k\in \ZZ$ and Theorem \ref{th:AkA} is proved. \end{proof}

Therefore, Theorem \ref{th:yk} is proved.
\end{proof}

And, ultimately, Theorem \ref{th:kontsevich} is proved. \endproof
\begin{example}
Let $r_1=r_2=2$. We have: $y_{k+1}zy_{k-1}=y_k^2+1,~y_{k-1}y_{k+1}z=y_kzy_kz + 1$
for all $k\in \ZZ$. This implies:
$$y_4z=y_2^{-1}(y_3 z y_3z+ 1)=(zy_0-y_2^{-1}y_1^2)y_3 (z y_3z)+y_2^{-1}$$
$$=zy_0y_3z y_3z-y_2^{-1}(y_1(y_1y_3z)y_3z-1)$$
Note that $y_1(y_1y_3z)y_3z-1=y_1(y_2zy_2z+ 1)y_3z-1=y_1y_2zy_2zy_3z+ y_1y_3z-1=y_2zy_1zy_2zy_3z+ (y_2z)^2$.
Therefore, $$y_4z=zy_0(y_3z)^2-(zy_1zy_2zy_3z+ zy_2z).$$
%Hence, 
%$$y_5=zy_1(y_4z)^2-(zy_2zy_3zy_4z+ y_3z)=zy_1(zy_0y_3z y_3z-(zy_1zy_2zy_3z+ y_2z))^2-(zy_2zy_3z(zy_0(y_3z)^2-(zy_1zy_2zy_3z+ y_2z))+ y_3z)$$
\end{example}

The noncommutative cluster algebra ${\mathcal A}={\mathcal A}(r_1,r_2)$ has a number symmetries in addition to the anti-involution $\sigma:{\mathcal A}(r_1,r_2)\widetilde \to {\mathcal A}(r_2,r_1)$ from Lemma \ref{le:reflection}: the translation  $y_k\mapsto y_{k+1}$, $k\in \ZZ$ defines an isomorphism $\tau:{\mathcal A}(r_1,r_2)\widetilde \to {\mathcal A}(r_2,r_1)$, which is an automorphism when $r_1=r_2$. 

\medskip

We conclude with a brief discussion of the presentation of ${\mathcal A}$. 

%It is easy to show that the following relations holds in ${\mathcal F}_2$ 

\begin{proposition} 
\label{pr:presentation}
The generators $y_0,y_1,y_2,y_3,z^{\pm 1}$ of  ${\mathcal A}$  satisfy (for $i=0,1,2$, $j=1,2$):
$$y_iy_{i+1}=y_{i+1}zy_i, y_{j+1}zy_{j-1}=y_j^{r_j}+1,~y_{j-1}y_{j+1}z=(y_jz)^{r_j}+1,~y_3zy_0-zy_0y_3z=y_2^{r_2-1}y_1^{r_1-1}-z(y_1z)^{r_1-1}(y_2z)^{r_2-1}$$

\end{proposition}

\begin{proof} Only the last relation needs to be proved (the first three relations are \eqref{eq:z-commutation}, \eqref{eq:cluster recursion}, and \eqref{eq:cluster recursion2} respectively). Indeed, using the available relations in ${\mathcal F}_2$, we obtain: 
$$y_0y_3z=((1+(y_1z)^{r_1})z^{-1}y_2^{-1})(y_1^{-1}(1+(y_2z)^{r_2}))=(1+(y_1z)^{r_1})z^{-1}y_1^{-1}z^{-1}y_2^{-1}(1+(y_2z)^{r_2})=h_{r_1}(y_1z)h_{r_2}(y_2z)\ ,$$
where $h_r(y)=y^{-1}+y^{r-1}$. Similarly, 
$$y_3zy_0=((1+y_2^{r_2})y_1^{-1})(z^{-1}y_2^{-1}(1+y_1^{r_1}))=(1+y_2^{r_2})y_2^{-1}y_1^{-1}(1+y_1^{r_1})=h_{r_2}(y_2)h_{r_1}(y_1)$$
Taking into account that $y_1y_2y_1^{-1}=y_2z$ and $y_2^{-1}y_1y_2=zy_1$, we obtain:
$$y_3zy_0=y_2^{r_2-1}y_1^{r_1-1}+h_{r_2}(y_2)y_1^{-1}+y_2^{-1}y_1^{r_1-1}+y_2^{-1}y_1^{-1}=y_2^{r_2-1}y_1^{r_1-1}+(zy_1)^{r_1-1}y_2^{-1}+y_1^{-1}h_{r_2}(y_2z)+y_1^{-1}z^{-1}y_2^{-1}$$ 
$$=y_2^{r_2-1}y_1^{r_1-1}+z(y_1z)^{r_1-1}(y_2z)^{-1}+z(y_1z)^{-1}h_{r_2}(y_2z)+z(y_1z)^{-1}(y_2z)^{-1}=y_2^{r_2-1}y_1^{r_1-1}+z(y_1z)^{r_1-1}(y_2z)^{r_2-1}\ .$$
The proposition is proved. 
\end{proof}

We expect that the relations in Proposition \ref{pr:presentation} are defining.
 
\bigskip

\noindent {\bf Acknowledgments}. This work started when the authors were visiting
IHES in July 2010. We thank Maxim Kontsevich for his kind hospitality and stimulating discussions.


\begin{thebibliography}{xxx}

%\bibitem {BR1} A.~Berenstein, V.~Retakh, Lie algebras and Lie groups over noncommutative rings, 
%{\it Advances in Mathematics}, Vol. 218, {\bf 6}, (2008),  pp.~1723--1758.

%\bibitem {BR2} A.~Berenstein, V.~Retakh, Noncommutative surfaces, triangulations and clusters, {\it preprint} (will be available for circulation in November 2010). 


\bibitem{BFZ-cluster} A.~Berenstein,  S.~Fomin, A.~Zelevinsky, Cluster
algebras III: Upper and lower bounds, {\em Duke Math. Journal}, vol.~ 126,
{\bf 1} (2005), pp.~1--52.


\bibitem {BZ10} A.~Berenstein,  A.~Zelevinsky, Quantum cluster algebras,
{\em Advances in Mathematics}, vol. 195, {\bf 2} (2005), pp.~405--455.



\bibitem{KeDiF4} P.~Di Francesco, R.~Kedem, Discrete non-commutative integrability: Proof of a conjecture by M.
Kontsevich, {\em Intern. Math. Res. Notes}, (2010) doi:10.1093/imrn/rnq024). 
%{\it preprint} arXiv:0909.0615, 2009.




%\bibitem{FZ1} S. Fomin, A. Zelevinsky, Cluster algebras~I: Foundations,
%\textsl{J.~Amer.\ Math.\ Soc.} \textbf{15} (2002), 497--529.

%\bibitem{FZ-laurent} S. Fomin, A. Zelevinsky,  The Laurent phenomenon. {\em
%Adv. in Appl. Math.} {\bf 28} (2002), no. 2, pp.~119--144.


%\bibitem{lewin} J.~Lewin; T.~Lewin, An embedding of the group algebra of a torsion-free one-relator group in a field. {\it J. Algebra} {\it 52} (1978), no. 1, pp. 39--74.


\bibitem{Cohn} P.M.~Cohn, Free rings and their relations, second edition, {\it Academic Press}, London, 1985.


\bibitem{Usnich} A.~Usnich, Non-commutative cluster mutations, {\it Doklady of the National Academy of
Sciences of Belarus}, {\bf 53}(4), 2009, pp.~27–-29.

\bibitem{Usnich1} A.~Usnich, Non-commutative Laurent phenomenon for two variables, {\it preprint} arXiv:1006.1211, 2010.

%\bibitem{Vakil} R.~Vakil, A geometric Littlewood-Richardson rule (with an appendix joint with A. Knutson), {\it Annals of Math.} {\bf 164} (2006), pp.~371-422. 



%\bibitem{VarVas} M.~Varagnolo, E.~Vasserot, Finite dimensional
%representations of DAHA and affine Springers fibers: the spherical case,
%{\it Duke Math. J.} {\bf 147} (2009), no. 3, pp.~439--540.

%\bibitem{Wo} S.~Woronowicz, Differential calculus on compact matrix
%pseudogroups (quantum groups). {\em Comm.\ Math.\ Phys.}  {\bf 122} (1989),
%no. 1, pp.~125--170.


%\bibitem{yakimov4} M.~Yakimov, Cyclicity of Lusztig's stratification of Grassmannians and Poisson geometry, 
%{\it Noncommutative structures in mathematics and physics}, Royal Flemish Acad. of Belgium for Sciences and Arts, 2010, pp.~258--262, 

%\bibitem{Y} D.-N.~Yetter, Quantum groups and representations of monoidal
%categories. \textsl{Math.\ Proc.\ Cambridge Philos.\ Soc.}\  108  (1990),
%no. 2, pp.~261--290.

%\bibitem{zel} A.~Zelevinsky, Connected components of real double Bruhat
%cells, {\em Internat. Math. Res. Notices} 2000, no. 21, pp.~1131--1154.

%\bibitem {zw} S. Zwicknagl, $R$-Matrix Poisson Algebras and Their
%Deformations, {\it Advances in Mathematics}, Volume 220, {\bf 1}, pp.~1--58.

 \end{thebibliography}
\end{document}